\documentclass[12pt]{article}
\usepackage{amsmath}
\usepackage{amssymb}
\usepackage{stmaryrd}
\usepackage{amsthm}

\usepackage[colorlinks=true,citecolor=black,linkcolor=black,urlcolor=blue]{hyperref}

\usepackage[dvipsnames]{pstricks} % Use the 68 standard colors known to dvips!
\usepackage{multido}

\usepackage{graphicx}
\usepackage{xcolor}

\newcommand{\showgrid}{}

\newcommand{\gridon}{\renewcommand{\showgrid}{\psset{subgriddiv=1,griddots=10,gridlabels=6pt}\psgrid}}

\gridon % enable grid if uncommented

\newgray{hellgrau}{.89}

\usepackage[applemac]{inputenc}

\usepackage{algorithmic}

\newif\ifenglish

\englishtrue

\newif\ifvariant
\variantfalse

\def\bit{\begin{itemize}}
\def\eit{\end{itemize}}
\def\beq{\begin{equation}}
\def\eeq{\end{equation}}

\englishtrue
\def\of#1{\left(#1\right)} % Klammerung von Argumenten einer Funktion
\def\pas#1{\left(#1\right)} % Klammerung mit normalen Klammern
\def\brk#1{\left[#1\right]} % Klammerung mit eckigen Klammern
\def\defeq{\stackrel{\text{\tiny def}}{=}}

\def\N{{\mathbb N}}

\def\0{{\mathbf 0}} % Nullelement in algebraischer Struktur
\def\1{{\mathbf 1}} % Einselement in algebraischer Struktur

\def\floor#1{\left\lfloor #1\right\rfloor}

\def\ceil#1{\left\lceil #1\right\rceil}

\def\CT1{CT1}
\def\xxxCT1{CT1}

\newtheorem{thm}{\ifenglish Theorem\else Satz\fi}[section] % "Master-Theorem" für Numerierung

\newtheoremstyle{excstyle}% name of the style to be used
  {1em}% measure of space to leave above the theorem. E.g.: 3pt
  {2pt}% measure of space to leave below the theorem. E.g.: 3pt
  {\sffamily\footnotesize\slshape}% name of font to use in the body of the theorem
  {0pt}% measure of space to indent
  {\sffamily\footnotesize\bfseries}% name of head font
  {:}% punctuation between head and body
  { }% space after theorem head
  {}% Manually specify head
\theoremstyle{excstyle}

\def\figref#1{\ifenglish Figure\else Abbildung\fi~\ref{#1}}

\ifenglish\else\fi

\def\figref#1{Fi\-gu\-re~\ref{#1}}
\def\defeq{:=}

\def\shift{{\mathbf E}}
\def\idop{{\mathbf I}}
\def\detdim#1{{\det_{#1}}}

\newgray{mfgray85}{0.85}

\newgray{mfgray70}{0.70}

\newgray{mfgray60}{0.60}

\title{Enumeration of symmetric Gelfand--Tsetlin patterns by linear algebra}

\author{Markus Fulmek\thanks{
Research supported by the National Research Network ``Analytic
Combinatorics and Probabilistic Number Theory'', funded by the
Austrian Science Foundation. 
}\\
\small Fakult\"at f\"ur Mathematik \\
\small Oskar-Morgenstern-Platz 1, A-1090 Wien, Austria \\
\small \tt Markus.Fulmek@Univie.Ac.At
}

\def\secA{\section}

\def\EM#1{{\em #1\/}}
\begin{document}
% "Normale" Bibliographie
\bibliographystyle{plain}
% Titelseite und Vorwort

\maketitle

\begin{abstract}
We shall present a ``linear algebraic'' proof (involving some calculations in the algebra
of linear operators on a vector space of polynomials and some manipulations of determinants)
of the formula for the enumeration of symmetric Gelfand--Tsetlin patterns
with fixed bottom row, which was proved by Tri Lai in the context of enumerating symmetric
lozenge tilings of a ``halved'' hexagon with ``dents''.
\end{abstract}

\secA{Gelfand--Tsetlin patterns}
A \EM{Gelfand--Tsetlin pattern} is a (finite) triangular array of natural numbers
$$
\begin{matrix}
 &  &  &  & u_{1,1} &  &  &  & \\
 &  &  & u_{2,1} & & u_{2,2} &  &  & \\
 &  & u_{3,1} & & u_{3,2} & & u_{3,3} &  & \\
 & u_{4,1} & & u_{4,2} & & u_{4,3} & & u_{4,4} & \\
\dots & & \dots & & \dots & & \dots & & \dots\\
\end{matrix}
$$
where the entries % corresponding to 
in row $i-1$ are in the following sense
``interlaced'' with the entries % corresponding to 
in row $i$:
\begin{equation}
\label{eq:interlacing}
u_{i,1}\leq u_{i-1,1}<u_{i,2}\leq u_{i-1,2} < u_{i,3}\leq\cdots u_{i-1,i-1}<u_{i,i}.
\end{equation}
For instance, the following array is a Gelfand--Tsetlin pattern with $5$ rows:
$$
\begin{matrix}
 &  &  &  & 9 &  &  &  & \\
 &  &  & 8 & & 10 &  &  & \\
 &  & 6 & & 9 & & 11 &  & \\
 & 4 & & 8 & & 10 & & 13 & \\
2 & & 7 & & 10 & & 11 & & 17\\
\end{matrix}
$$
The enumeration of Gelfand--Tsetlin patterns with fixed bottom row $\pas{u_{k_1},\dots,u_{k,k}}$
is given by the simple product formula (see \cite[Proposition 2.1]{CohnLarsenPropp:1998:TSOATBPP},
where a very concise and elegant proof is given)
$$
\prod_{1\leq i<j\leq k}\frac{u_j-u_i}{j-i}.
$$

\section{Symmetric Gelfand--Tsetlin patterns}
\def\GTP#1#2{\brk{{#1}_{i,j}}_{1,1}^{#2,i}}
\def\halfGTP#1#2{\brk{{#1}_{i,j}}_{1,1}^{#2,\floor{i/2}}}
Let us call a Gelfand--Tsetlin pattern $U=\GTP{u}{n}$
with $n$ rows (counted from the top) \EM{symmetric} if ``the entries are symmetric with
respect to the vertical central axis'' (see the left picture in \figref{fig:encoding-symmetric}), i.e., if
$$
u_{i,j} = \pas{2 u_{1,1} - 1} + i - u_{i,i-j+1} \text{ for all } j=1,2,\dots,i
$$
holds for \EM{all} rows $i$ of $U$ (note that this condition \EM{always} holds for $i=1$).
For a symmetric Gelfand--Tsetlin pattern $U$,
the middle entries are
uniquely determined\footnote{Here,
$\ceil z$ and $\floor z$ denote the unique integers such that $\ceil z - 1 < z \leq \ceil z$
and $\floor z \leq z < \floor z+1$, respectively.} in all \EM{odd} rows $i$
$$
u_{i,\ceil{i/2}}=u_{1,1}+\floor{\frac{i-1}2},
$$
while in all \EM{even} rows $i$ we must have
$$
u_{i,\ceil{i/2}+1} > u_{1,1}+\floor{\frac{i-1}2}.
$$
So a symmetric Gelfand--Tsetlin pattern $U$ is uniquely determined by the entry $u_{1,1}$ and the
entries corresponding to the ``right half'' (see the right picture in \figref{fig:encoding-symmetric}):
$$
u_{i,\ceil{i/2}+1},u_{i,\ceil{i/2}+2},\dots,u_{i,i}\text{ for } i=2,3,\dots n.
$$
For convenience, we shall ``shift and reverse'' the entries corresponding to the  ``right half'', i.e., we change the notation as follows:
\begin{equation}
\label{eq:change-notation}
x_{i,j} \defeq u_{i,\ceil{i/2}+1-j} - 
\pas{u_{1,1}+\floor{\frac{i-1}2}}
\text{ for }i=2,3,\dots n \text{ and }j=1,2,\dots,\floor{\frac i2},
\end{equation}
and observe that the \EM{enumeration} of symmetric Gelfand--Tsetlin patterns with fixed bottom row
$$\pas{u_{n,1},u_{n,2},\dots,u_{n,n}}$$
does not depend on $u_{1,1}$, but amounts
to the enumeration of arrays $\halfGTP{x}{n}$ with fixed bottom row
$$\pas{x_{n,1},x_{n,2},\dots,x_{n,\floor{\frac n2}}}$$
where rows $i$ are ``interlaced'' as follows:
\begin{align}
0 < x_{i-1,\floor{i/2}} \leq x_{i,\floor{i/2}} < x_{i-1,\floor{i/2}-1} \leq x_{i,\floor{i/2}-1} < \dots < x_{i-1,1} \leq x_{i,1} &\text{ for odd } i,\label{eq:halfGTP-interlacing-odd}\\ 
x_{i,\floor{i/2}} \leq x_{i-1,\floor{i/2}-1} < x_{i,\floor{i/2}-1} \leq x_{i-1,\floor{i/2}-2} < \dots \leq x_{i-1,1} < x_{i,1}&\text{ for even } i.\label{eq:halfGTP-interlacing-even}
\end{align}
We shall call such array $\halfGTP{x}{n}$ a \EM{halved Gelfand--Tsetlin pattern}; note that its
first row is always \EM{empty}.
\begin{figure}
\setcounter{MaxMatrixCols}{20}
{\small
$$
\begin{matrix} 
 &  &  &  &  &  &  & {\gray 4} &  &  &  &  &  &  & \\
 &  &  &  &  &  & {\gray 3} & & 6 &  &  &  &  &  & \\
 &  &  &  &  & {\gray 2} & & {\gray 5} & & 8 &  &  &  &  & \\
 &  &  &  & {\gray 2} & & {\gray 4} & & 7 & & 9 &  &  &  & \\
 &  &  & {\gray 1} & & {\gray 3} & & {\gray 6} & & 9 & & 11 &  &  & \\
 &  & {\gray 1} & & {\gray 3} & & {\gray 5} & & 8 & & 10 & & 12 &  & \\
 & {\gray 1} & & {\gray 3} & & {\gray 4} & & {\gray 7} & & 10 & & 11 & & 13 & \\
{\gray 1} & & {\gray 2} & & {\gray 4} & & {\gray 7} & & 8 & & 11 & & 13 & & 14\\
\end{matrix}
\quad\quad
\quad\quad
\begin{matrix}
 {\gray 0} &  &  &  &  &  &  & \\
  & 2 &  &  &  &  &  & \\
 {\gray 0} & & 3 &  &  &  &  & \\
 & 2 & & 4 &  &  &  & \\
 {\gray 0} & & 3 & & 5 &  &  & \\
  & 2 & & 4 & & 6 &  & \\
 {\gray 0} & & 3 & & 4 & & 6 & \\
 & 1 & & 4 & & 6 & & 7\\
\end{matrix}
$$
The symmetric Gelfand--Tsetlin pattern with $8$ rows shown to the left is ``encoded'' by
the single entry in the first row ($4$, which uniquely determines the \EM{central} entries in all odd rows) and its shifted ``right half'' (shown to the right), where
we introduced starting entries $0$ in odd rows just to make clear the connection between the arrays:
The array to the right is the ``right half'' of the array to the left, where we subtracted
\begin{itemize}
\item the row's central entry (for odd rows), % or 
\item the central entry of the row above (for even rows)
\end{itemize}
from the entries. The change of notation  (see \eqref{eq:change-notation})
amounts to reading these entries from the right;
i.e., the bottom row of the ``right half'' is denoted as $\pas{7,6,4,1}$.
}
\caption{Illustration: Encoding of symmetric Gelfand--Tsetlin patterns.}
\label{fig:encoding-symmetric}
\end{figure}

\section{Enumeration of symmetric Gelfand--Tsetlin patterns with fixed bottom row}
Let us denote the number of halved Gelfand--Tsetlin patterns with fixed bottom row
$\pas{x_1,x_2,\dots,x_k}$
\def\evenrow{E}
\def\oddrow{O}
\bit
\item by
$\oddrow\of{x_1,x_2,\dots,x_k}$ if the number of rows is $2k+1$,
\item and by
$\evenrow\of{x_1,x_2,\dots,x_k}$ if the number of rows is $2k$.
\eit
Denoting the empty row ($k=0$) by ``$-$'', we clearly have
$$
\oddrow\of{-} = \evenrow\of{x_1} = 1.
$$
From \eqref{eq:halfGTP-interlacing-odd} and \eqref{eq:halfGTP-interlacing-even} we
immediately obtain the following \EM{summation recursions}:
\begin{align}
\oddrow\of{x_1,x_2,\dots,x_k} &=
\sum_{y_k=1}^{x_k}\sum_{y_{k-1}=x_k+1}^{x_{k-1}}\cdots\sum_{y_1=x_2+1}^{x_1}
\evenrow\of{y_1,y_2,\cdots,y_k},\label{eq:odd-recursion}\\
\evenrow\of{x_1,x_2,\dots,x_k} &= 
\sum_{y_{k-1}=x_k}^{x_{k-1}-1}\sum_{y_{k-2}=x_{k-1}}^{x_{k-2}-1}\cdots\sum_{y_1=x_2}^{x_1-1}
\oddrow\of{y_1,y_2,\cdots,y_{k-1}}.\label{eq:even-recursion}
\end{align}
From these observations we see that %: %it follows immediately that
%\bit
%\item 
$\evenrow\of{x_1,x_2,\dots,x_k}$ and $\oddrow\of{x_1,x_2,\dots,x_k}$ (viewed as functions
of the variables) are actually 
\EM{polynomials} in $x_1,x_2,\dots,x_k$.
% \item d
% \eit
By direct computation we get the first instances
\begin{align*}
\oddrow\of{x_1} &= x_1,\\
\evenrow\of{x_1,x_2} &= \frac12\pas{x_1-x_2}\pas{x_2+x_1-1}, \\
\oddrow\of{x_1,x_2} &= \frac16x_1 x_2\pas{x_1-x_2}\pas{x_2+x_1}.
\end{align*}
After working out some more instances it is not hard to guess the factorization for these polynomials
(these factorizations are given in \cite[equations (3.1) and (3.2) of Theorem 3.1.]{Lai:2014:EOTOQAR}, where they are proved in the context of symmetric lozenge tilings
of a ``halved'' hexagon with ``dents''; i.e., missing triangles):

\begin{thm}
Let $\mathbf x\defeq\pas{x_1,x_2,\dots}$ be an infinite series of variables. Define
the polynomials
\begin{align}
o_k\of{\mathbf x} &\defeq
\frac{
\prod_{j=1}^kx_j\prod_{i=1}^{j-1}\pas{x_i-x_j}\pas{x_i+x_j}
}{
1!\cdot3!\cdots\pas{2k-1}!
},\label{eq:odd-formula}\\
e_k\of{\mathbf x} &\defeq
\frac{
\prod_{j=1}^k\prod_{i=1}^{j-1}\pas{x_i-x_j}\pas{x_i+x_j-1}
}{
0!\cdot2!\cdots\pas{2k-2}!
}. \label{eq:even-formula}
\end{align}
Then we have
$\oddrow\of{x_1,x_2,\dots,x_k} = o_k\of{\mathbf x}$ and %,\\ %\label{eq:odd-formula}\\
$\evenrow\of{x_1,x_2,\dots,x_k} = e_k\of{\mathbf x}$.
\end{thm}

\begin{proof}
We already saw that the assertion is true for $k\leq 2$; so it suffices to show that the
polynomials $o_k$ and $e_k$ obey the summation recursions \eqref{eq:odd-recursion} and
\eqref{eq:even-recursion}. We shall show this by a bit of linear algebra:
On the real vector space
of all polynomials in $\mathbf x$, we define the \EM{identity operator}
$$
\idop: p\mapsto p
$$
and the \EM{shift operators}
$$
\shift_i: p\of{x_1,x_2,\dots,x_i,x_{i+1}\dots} \mapsto p\of{x_1,x_2,\dots,x_i+1,x_{i+1}\dots},
$$
with the obvious \EM{inverses}
$$
\shift_i^{-1}: p\of{x_1,x_2,\dots,x_i,x_{i+1}\dots} \mapsto p\of{x_1,x_2,\dots,x_i-1,x_{i+1},\dots}.
$$
Clearly, the operators $\idop,\shift_1,\shift_1^{-1},\shift_2,\shift_2^{-1},\dots$ are \EM{linear}
and \EM{pairwise commutative}. % So we want to prove the following \EM{translation} of 
Translating the
summation recursions \eqref{eq:odd-recursion} and
\eqref{eq:even-recursion} into \EM{operator language}, we must prove: For every
sequence of variables $\pas{x_1,x_2,\dots,x_k,0,\dots}$ such that $x_i-x_{i+1}\in\N$
for $i=1,2,\dots,k$ there holds
\begin{align}
o_k\of{x_1,x_2,\dots,x_k,0,\dots} &=
\pas{\prod_{i=1}^k\sum_{j=1}^{x_i-x_{i+1}}\shift_i^j}
e_k\of{x_2,x_3,\cdots,x_{k},0\dots},\label{eq:odd-recursion2}\\
e_k\of{x_1,x_2,\dots,x_k,0,\dots} &= 
\pas{\prod_{i=1}^{k-1}\sum_{j=0}^{x_{i}-x_{i+1}-1}\shift_i^j}
o_{k-1}\of{x_2,x_3,\cdots,x_{k},0,\dots}.\label{eq:even-recursion2}
\end{align}
Equations \eqref{eq:odd-recursion2} and \eqref{eq:even-recursion2} can be deduced
from their ``inverse relations'', namely the \EM{difference equations} (which we shall show below)
\begin{align}
\pas{\prod_{i=1}^k\pas{\idop-\shift_i^{-1}}} o_k\of{\mathbf x} &= e_k\of{\mathbf x} \label{eq:diff-odd-even}, \\
\pas{\prod_{i=1}^{k-1}\pas{\shift_i-\idop}} e_k\of{\mathbf x} &= o_{k-1}\of{\mathbf x} \label{eq:diff-even-odd}
\end{align}
by simple computations in the operator algebra; we start with \eqref{eq:odd-recursion2}:
\begin{align}
&\phantom{=}
\pas{\prod_{i=1}^k\sum_{j=1}^{x_i-x_{i+1}}\shift_i^j}
\underbrace{
\pas{\prod_{i=1}^k\pas{\idop-\shift_i^{-1}}}
o_k\of{x_2,x_3,\cdots,x_{k},0\dots}
}_{
e_k\of{x_2,x_3,\cdots,x_{k},0\dots}\text{ by \eqref{eq:diff-odd-even}}
} \notag \\
&=
\pas{\prod_{i=1}^k\pas{\shift_i^{x_i-x_{i+1}}-\idop}}
o_k\of{x_2,x_3,\cdots,x_{k},0\dots}.\label{eq:odd-comp}
\end{align}
Now observe that $o_k$ is zero whenever among the first $k$ variables some
variable is zero \EM{or}
two (consecutive) variables are equal: So in the expansion of \eqref{eq:odd-comp},
there is a sole non--zero term, namely
$$
\pas{\prod_{i=1}^k{\shift_i^{x_i-x_{i+1}}}}
o_k\of{x_2,x_3,\cdots,x_{k},0\dots} = o_k\of{x_1,x_2,\cdots,x_{k-1},x_k,0\dots}.
$$

Similarly, we compute for \eqref{eq:even-recursion2} (note that the $k$--th variable
in $e_k$ can be chosen \EM{arbitrarily} in \eqref{eq:diff-even-odd}; we choose
it to be equal to the variable in position $k-1$):
\begin{align}
&\phantom{=}
\pas{\prod_{i=1}^{k-1}\sum_{j=0}^{x_{i}-x_{i+1}-1}\shift_i^j}
\underbrace{
\pas{\prod_{i=1}^{k-1}\pas{\shift_i-\idop}}
e_{k}\of{x_2,x_3,\cdots,x_{k},x_k,0,\dots}
}_{
o_{k-1}\of{x_2,x_3,\cdots,x_{k},0,\dots}\text{ by \eqref{eq:diff-even-odd}}
} \notag \\
&=
\pas{\prod_{i=1}^{k-1}\pas{\shift_i^{x_i-x_{i+1}}-\idop}}
e_k\of{x_2,x_3,\cdots,x_{k},x_k,0,\dots}. \label{eq:even-comp}
\end{align}
As before, there is a sole non--zero term in the expansion of \eqref{eq:even-comp}
since $e_k$ is zero whenever two (consecutive) variables among the first $k$ variables are equal, namely
$$
\pas{\prod_{i=1}^{k-1}{\shift_i^{x_i-x_{i+1}}}}
e_k\of{x_2,x_3,\cdots,x_{k},x_k,0,\dots} = e_k\of{x_1,x_2,\cdots,x_{k-1},x_k,0,\dots}.
$$

So it remains to prove \eqref{eq:diff-odd-even} and \eqref{eq:diff-even-odd}:
Let us denote by $\detdim{k}\of{a_{i,j}}$ the \EM{determinant} of some matrix
$\pas{a_{i,j}}_{\pas{i,j}=\pas{1,1}}^{\pas{k,k}}$.
Then by the well--known \EM{Vandermonde identity}% for determinants
$$
\detdim{k}\of{x_i^{j-1}}=\prod_{1\leq i<j\leq k}\pas{x_j-x_i}
$$
we obtain % (after some obvious simple manipulations):
\begin{align}
o_k\of{\mathbf x} &=
\detdim{k}\of{\frac{x_i^{2k-2j+1}}{\pas{2k-2j+1}!}},  \label{eq:odd-det}\\
e_k\of{\mathbf x} &=
\detdim{k}\of{\frac{\pas{x_i-\frac{1}{2}}^{2 k-2 j}}{\pas{2 k-2 j}!}}. \label{eq:even-det}
\end{align}
By \EM{linearity} of the identity operator and the shift operators, \eqref{eq:diff-odd-even} and
\eqref{eq:diff-even-odd} are equivalent to the following \EM{determinantal identities}, where we
took into account that
the operators in \eqref{eq:diff-even-odd} do not affect the $k$--th row in \eqref{eq:even-det}:
\begin{align}
\detdim{k}\of{\frac{x_i^{2k-2j+1}-\pas{x_i-1}^{2k-2j+1}}{\pas{2k-2j+1}!}}
&=
\detdim{k}\of{\frac{\pas{x_i-\frac{1}{2}}^{2k-2j}}{\pas{2k-2j}!}},
\label{eq:detid-odd-even} \\
\detdim{k}\of{
\left.
\underbrace{
\frac{\pas{x_i+\frac{1}{2}}^{2k-2j}-\pas{x_i-\frac{1}{2}}^{2k-2j}}%
{\pas{2k-2j}!}
}_{\text{row }i<k}\;
\right\vert
\underbrace{
\frac{\pas{x_k-\frac{1}{2}}^{2k-2j}}%
{\pas{2k-2j}!}
}_{\text{row $k$}}
}
&=
\detdim{k-1}\of{\frac{x_i^{2k-2j+1}}{\pas{2k-2j+1}!}}. \label{eq:detid-even-odd}
\end{align}
Clearly, the left--hand side of \eqref{eq:detid-even-odd} is equal to
the minor of the first $k-1$ rows and columns, so \eqref{eq:detid-even-odd} is
(by reversing the order of columns to simplify notation) equivalent to
\begin{equation}
\detdim{k-1}\of{
\frac{\pas{x_i+\frac{1}{2}}^{2 j}-\pas{x_i-\frac{1}{2}}^{2 j}}%
{\pas{2 j}!}
}
=
\detdim{k-1}\of{\frac{x_i^{2j-1}}{\pas{2j-1}!}}. \label{eq:detid-even-odd2}
\end{equation}
Substituting $x_i\to y_i+\frac12$ in \eqref{eq:detid-odd-even} (and again reversing the order of columns to simplify notation) gives the \EM{equivalent} determinantal identity
\begin{equation}
\detdim{k}\of{\frac{\pas{y_i+\frac12}^{2j-1}-\pas{y_i-\frac12}^{2j-1}}{\pas{2j-1}!}}
=
\detdim{k}\of{\frac{{y_i}^{2 j-2}}{\pas{2 j-2}!}} \label{eq:detid-odd-even-subs},
\end{equation}
and both identities \eqref{eq:detid-even-odd2} and \eqref{eq:detid-odd-even-subs} follow
from the fact that the matrices corresponding to %the determinants on 
the left--hand sides
can be transformed to the matrices corresponding to the
% determinants on the 
right--hand sides by elementary (determinant--preserving)
column operations, since the \EM{leading terms} in the expansions
of the left--hand side's entries are equal to the corresponding
entries of the right--hand sides: % are
$$
\frac{m\cdot z^{m-1}\pas{\frac12+\frac12}}{{m}!} = \frac{z^{m-1}}{\pas{m-1}!}
\text{ for } m=2j \text{ or } m=2j-1.
$$
This finishes the proof.
\end{proof}

%%%%%%%%%%%%%%%%%%%%%%%%%%%%%%%%%%%%%%%%%%%%%%%%%%%%%%%%%%%%%%%%%%%%%
%  Bibliographie-File einlesen:                                     %
%%%%%%%%%%%%%%%%%%%%%%%%%%%%%%%%%%%%%%%%%%%%%%%%%%%%%%%%%%%%%%%%%%%%%
\bibliography{/Users/mfulmek/Work/TeX/database}

\end{document}